\documentclass[11pt, reqno]{amsart}
\usepackage{amsmath, amsthm, amscd, amsfonts, amssymb, graphicx, color}
\usepackage[bookmarksnumbered, colorlinks, plainpages]{hyperref}

\input{mathrsfs.sty}

\newtheorem{theorem}{Theorem}[section]
\newtheorem{lemma}[theorem]{Lemma}

\newtheorem{corollary}[theorem]{Corollary}
\theoremstyle{definition}

\theoremstyle{remark}
\newtheorem{remark}[theorem]{Remark}
\numberwithin{equation}{section}
\begin{document}

\title[Reverses of the Young inequality ] {Reverses of the Young inequality for matrices and operators}

\author[M. Bakherad, M. Krni\' c, M.S. Moslehian]{Mojtaba Bakherad$^1$, Mario Krni\' c$^2$ and Mohammad Sal Moslehian$^3$}

\address{$^{1}$ Department of Pure Mathematics, Ferdowsi University of Mashhad, P. O. Box 1159, Mashhad 91775, Iran.}
\email{mojtaba.bakherad@yahoo.com; bakherad@member.ams.org}

\address{$^{2}$ University of Zagreb, Faculty of Electrical Engineering and Computing, Unska 3, 10000 Zagreb, Croatia.}
\email{mario.krnic@fer.hr}

\address{$^3$ Department of Pure Mathematics, Center of Excellence in
Analysis on Algebraic Structures (CEAAS), Ferdowsi University of
Mashhad, P. O. Box 1159, Mashhad 91775, Iran}
\email{moslehian@um.ac.ir,
moslehian@member.ams.org}


\begin{abstract}
We present some reverse Young-type inequalities for the Hilbert-Schmidt norm as well as any unitarily invariant norm. Furthermore, we give some inequalities dealing with operator means. More precisely, we show that if $A, B\in {\mathfrak B}(\mathcal{H})$ are positive operators and $r\geq 0$, $A\nabla_{-r}B+2r(A\nabla B-A\sharp B)\leq A\sharp_{-r}B$ and prove that equality holds if and only if $A=B$. We also establish several reverse Young-type inequalities involving trace,
determinant and singular values. In particular, we show that if
$A, B$ are positive definite matrices and $r\geq 0$,
then $\label{reverse_trace} \mathrm{tr}((1+r)A-rB)\leq
\mathrm{tr}\left|A^{1+r}B^{-r} \right|-r\left( \sqrt{\mathrm{tr}
A} - \sqrt{\mathrm{tr} B}\right)^{2}$.
\end{abstract} \maketitle
\section{Introduction and preliminaries}

Let $\mathcal{H}$ be a Hilbert space and let $\mathfrak{B}(\mathcal{H})$ be the $C^*$-algebra of all bounded linear operators on $\mathcal{H}$ with the operator norm $\|\cdot\|$ and the identity $I_{\mathcal{H}}$. If $\dim \mathcal{H}=n$, then we identify $\mathfrak{B}(\mathcal{H})$ with the space $\mathcal{M}_n$ of all
$n\times n$ complex matrices and denote the identity matrix by $I_n$. For an operator $A \in
\mathfrak{B}(\mathcal{H})$, we write $A\geq 0$ if $A$ is positive (positive semidefinite for matrices), and
$A>0$ if $A$ is positive invertible (positive definite for matrices). For $A, B \in \mathfrak{B}(\mathcal{H})$, we say $A\geq B$ if $A-B\geq0$. Let $\mathfrak{B}^{+}(\mathcal{H})$ (resp., $\mathcal{P}_n$) denote the set of all positive
invertible operators (resp., positive definite matrices).
\noindent A norm $|||\,.\,|||$ on $\mathcal{M}_n$ is called unitarily invariant norm if
$|||UAV|||=|||A|||$ for all $A\in\mathcal{M}_n$ and all unitary
matrices $U, V\in\mathcal{M}_n$. The Hilbert-Schmidt norm is defined
by $\|A\|_2=\left(\sum_{j=1}^ns_j^2(A)\right)^{1/2}$, where $s(A)=(s_1(A),\cdots, s_n(A))$
denotes the singular values of $A$, that is, the eigenvalues of the positive semidefinite matrix $|A|=(A^*A)^{1/2}$, arranged in the decreasing order with their multiplicities counted.
This norm is unitarily invariant. It is known that if $A=[a_{ij}]\in\mathcal{M}_n$, then $\|A\|_2=\Big{(}\sum_{i,j=1}^n|a_{ij}|^2\Big{)}^{1/2}$.

The weighted operator
arithmetic mean $\nabla_\nu$, geometric mean $\sharp_\nu$, and
harmonic mean $!_\nu$, for $ \nu\in[0,1]$ and $A, B\in
\mathfrak{B}^{+}(\mathcal{H})$,
 are defined as follows:
\begin{equation*}
 A\ \! \nabla_\nu \ \! B=(1-\nu)A+\nu B,
\end{equation*}
\begin{equation*}
A \ \!\sharp_\nu \ \!B=A^\frac{1}{2}\big(A^{-\frac{1}{2}}B
A^{-\frac{1}{2}}\big)^{\nu}A^\frac{1}{2},
\end{equation*}
\begin{equation*}
 A!_\nu B=\left((1-\nu)A^{-1}+\nu B^{-1}
\right)^{-1}.
\end{equation*}
If $\nu=1/2$, we denote arithmetic, geometric and harmonic mean,
respectively, by $\nabla$, $\sharp$ and $!$, for brevity.

The classical Young inequality states that
 \begin{align*}
 a^\nu b^{1-\nu}\leq \nu a+(1-\nu)b,
\end{align*}
when $a,b\geq0$ and $\nu\in[0,1]$. If $\nu={1\over2}$,
we obtain the arithmetic-geometric mean
inequality $\sqrt{ab}\leq {a+b\over2}.$
An operator Young inequality reads as follows:
\begin{equation}
\label{AGH_operator} A!_\nu B\leq A \ \!\sharp_\nu \ \!B\leq A\ \!
\nabla_\nu \ \! B, \quad \nu\in [0,1],
\end{equation}
where $A, B \in \mathfrak{B}^{+}(\mathcal{H})$ and $\nu\in[0,1]$; cf. \cite{furuta}. For other generalization of the Young inequality see \cite{MAN1, MAN2}. A matrix Young inequality
due to Ando \cite{ando} asserts that
\begin{align*}
 s_j(A^\nu B^{1-\nu})\leq s_j\left(\nu A+(1-\nu)B\right),
\end{align*}
in which $A, B\in\mathcal{M}_n$ are positive semidefinite,
$j=1,2,\ldots,n$, and $\nu\in[0,1]$. The above singular value
inequality entails the unitarily invariant norm inequality
 \begin{align*}
|||A^\nu B^{1-\nu}|||\leq |||\nu A+(1-\nu)B|||,
\end{align*}
where $A,B\in\mathcal{M}_n$ are positive semidefinite and
$0\leq\nu \leq1$. Kosaki \cite{kosa} proved that the inequality
\begin{align}\label{kosaki1}
 \|A^\nu X B^{1-\nu}\|_2\leq \|\nu AX+(1-\nu)XB\|_2
\end{align}
holds for matrices $A,B,X\in\mathcal{M}_n$ such that $A, B$ are
positive semidefinite, and  for $0\leq\nu \leq1$. It should be
mentioned here that for $\nu\neq{1\over2}$
 inequality \eqref{kosaki1} may not hold for other unitarily invariant norms.
 Hirzallah and Kittaneh \cite{hirzallah}, gave a
 refinement of \eqref{kosaki1} by showing that
\begin{align}\label{kosaki11}
 \|A^\nu X B^{1-\nu}\|_2^2+r_0^2\|AX-XB\|_2^2\leq \|\nu AX+(1-\nu)XB\|_2^2,
\end{align}
in which $A,B,X\in\mathcal{M}_n$ are such that $A, B$
are positive semidefinite, $0\leq\nu \leq1$ and $r_0=\min\{\nu, 1-\nu\}$.
 A determinant version of the Young inequality is also known (see \cite[p. 467]{horn}):
\begin{align*}
 {\rm det}(A^\nu B^{1-\nu})\leq{\rm det}(\nu A+(1-\nu)B),
\end{align*}
where $A,B,X\in\mathcal{M}_n$ are such that $A, B$
are positive semidefinite and $0\leq\nu \leq1$. This determinant
inequality was recently improved in \cite{young2}. Further,
Kittaneh \cite{kittaneh}, proved that
\begin{align}\label{edc1}
|||A^{1-\nu}XB^\nu|||\leq|||AX|||^{1-\nu}|||XB|||^{\nu},
\end{align}
in which  $|||\,.\,|||$ is any unitarily invariant norm, $A, B,
X\in\mathcal{M}_n$ are such that $A, B$ are positive semidefinite
and $0\leq\nu \leq1$. Conde \cite{conde}, showed that
\begin{align*}
2|||A^{1-\nu}XB^\nu|||+
\left(|||AX|||^{1-\nu}-|||XB|||^\nu\right)^2
\leq|||AX|||^{2(1-\nu)}+|||XB|||^{2\nu},
\end{align*}
where  $|||\,.\,|||$ is  unitarily invariant norm, $A, B,
X\in\mathcal{M}_n$ are such that $A, B$
are positive semidefinite and $0\leq\nu \leq1$. Tominaga \cite{TOM1, TOM2} employed Specht's ratio to Young inequality. In addition, some reverses of Young inequality are established in \cite{FUR}.

For $a, b\in\mathbb{R}$, the number  $x=\nu a+(1-\nu)b$ belongs to the interval $[a,b]$ for all $\nu\in[0,1]$, and
is outside the interval for all  $\nu>1$ or $\nu<0$. Exploiting this obvious fact,  Fujii
\cite{fuji},
showed that if $f$ is an operator concave function on an interval $J$, then the inequality
\begin{align*}
f(C^*XC-D^*YD)\leq |C|f(V^*XV)|C|-D^*f(Y)D
\end{align*}
holds for all self-adjoint operators $X,Y$ and operators $C, D$ in $\mathfrak{B}(\mathcal{H})$ with spectra in $J$, such that $C^*C-D^*D=I_\mathcal{H}$, $\sigma(C^*XC-D^*YD)\subseteq J$ and $C=V|C|$ is the polar decomposition of $C$.

In this direction, by using some numerical inequalities, we obtain
reverses of \eqref{AGH_operator}, \eqref{kosaki1},
 \eqref{kosaki11} and \eqref{edc1} under some mild conditions. We also aim to
give some reverses of the Young inequality dealing with operator
means of positive operators. Finally, we present some singular
value inequalities of Young-type involving trace and determinant.
\section{Reverses of the Young inequality for the Hilbert-Schmidt norm}

\noindent In this section we deal with reverses of the Young
inequality for the Hilbert-Schmidt norm.
 To this end, we need some lemmas.
\begin{lemma}\label{lemma11}
Let $a, b>0$. If $r \geq0$ or $r\leq-1$, then
\begin{align}\label{prva}
(1+r)a-rb\leq a^{1+r}b^{-r}.
 \end{align}
\end{lemma}
\begin{proof}
Let $f(t)=t^{-r}-(1+r)+rt$, $t\in(0,\infty)$. It is easy to see
that $f(t)$ attains its minimum at $t=1$, on the interval
$(0,\infty)$. Hence, $f(t)\geq f(1)=0$ for all $t>0$. Letting
$t={b\over a}$, we get the desired inequality.
\end{proof}
\begin{remark}
By virtue of Lemma \ref{lemma11},  it follows that the inequality
\begin{align}\label{bsd}
\left((1+r)a-rb\right)^2\leq\left(a^{1+r}b^{-r}\right)^2
 \end{align}
 holds if  $a\geq b>0$ and  $r\geq0$, or $b\geq a>0$ and $r\leq
 -1$.
\end{remark}
\begin{lemma}\cite[Theorem 3.4]{Zhang1}\label{shour}
(Spectral Decomposition) Let $A\in\mathcal{M}_n$ with eigenvalues
$\lambda_1, \lambda_2, \ldots, \lambda_n$. Then $A$ is normal if
and only if there exists a unitary matrix $U$ such that
\begin{align*}
U^*AU={\rm diag}(\lambda_1, \lambda_2, \cdots, \lambda_n).
 \end{align*}
 In particular, $A$ is positive definite if and only if  $\lambda_j>0$ for $j=1,2,\ldots ,n$.
\end{lemma}
 Now, our first result reads as follows.
\begin{theorem}\label{fri}
Let $A, B, X\in\mathcal{M}_n$ and let $m, m'$ be positive scalars.
If $A\geq mI_n\geq B>0$ and $r\geq0$, or $B\geq m'I_n\geq A>0$ and
$r\leq -1$, then the following inequality holds:
 \begin{align*}
\|(1+r)AX-rXB\|_2\leq \|A^{1+r}XB^{-r}\|_2.
 \end{align*}
 \end{theorem}
\begin{proof}
It follows from Lemma \ref{shour} that there are unitary matrices
$U, V\in\mathcal{M}_n$ such that $A=U\Lambda U^*$ and $B=V\Gamma
V^*$, where $\Lambda={\rm diag}(\lambda_1, \lambda_2, \cdots,
\lambda_n)$, $\Gamma={\rm diag}(\gamma_1, \gamma_2, \cdots,
\gamma_n)$, and $\lambda_j, \gamma_j$, $j=1,2 \ldots, n$, are
positive. If $Z=U^*XV=\big{[}z_{ij}\big{]}$, then
\begin{align}\label{311}
(1+r)AX-rXB =U\Big{(}(1+r)\Lambda Z-rZ\Gamma\Big{)}V^*
=U\Big{[}\Big{(}(1+r)\lambda_i-r\gamma_j\Big{)}z_{ij}\Big{]}V^*
 \end{align}
 and
\begin{align}\label{321}
A^{1+r}XB^{-r}=U\Lambda^{1+r}U^*XV\Gamma^{-r}V^*=
U\Lambda^{1+r}Z\Gamma^{-r}V^*=
U\Big{[}\Big{(}\lambda_i^{1+r}\gamma_j^{-r}\Big{)}z_{ij}\Big{]}V^*.
\end{align}
Suppose first that  $A\geq mI_n\geq B>0$ and $r\geq 0$. Then, it
follows that
\begin{align}\label{bigstar}
 \lambda_i\geq \gamma_j, \qquad 1\leq i,j\leq n,
 \end{align}
so, utilizing \eqref{311} and \eqref{321}, we have
\begin{align*}
\|(1+r)AX-rXB\|_2^2&=\sum_{i,j=1}^n \Big{(}(1+r)\lambda_i-r\gamma_j\Big{)}^2|z_{ij}|^2
\\&\leq\sum_{i,j=1}^n\Big{(}\lambda_i^{1+r}\gamma_j^{-r}\Big{)}^2|z_{ij}|^2\qquad \textrm{(by inequality \eqref{bsd} and \eqref{bigstar}})
\\&=\|A^{1+r}XB^{-r}\|_2^2.
 \end{align*}The same conclusion can be drawn for the case of $B\geq m'I_n\geq A>0$ and
$r\leq -1$.
\end{proof}
 Recall that a continuous real valued function $f$, defined on an interval $J$, is called operator monotone if
  $A\leq B$ implies $f(A)\leq g(B)$, for all $A, B\in\mathcal{M}_n$ with spectra in $J$.
  Now, the following result can be accomplished as an immediate
  consequence of Theorem \ref{fri}.
\begin{corollary}
Suppose that $A_j, B_j, X\in\mathcal{M}_n$, $1\leq j\leq n$, with
spectra in an interval $J$, and let $m_j, m_j'$, $1\leq j\leq n$,
be positive scalars. If $A_j\geq m_jI_n\geq B_j>0$, $1\leq j\leq
n$, and $r\geq 0$, or $B_j\geq m_j'I_n\geq A_j>0$, $1\leq j\leq
n$, and $r\leq -1$, then the inequality
 \begin{align*}
\left\|\sum_{j=1}^n\Big((1+r)f(A_j)X-rXf(B_j)\Big)\right\|_2\leq
\left\|\left(\sum_{j=1}^n
f(A_j)\right)^{1+r}X\left(\sum_{j=1}^nf(B_j)\right)^{-r}\right\|_2
 \end{align*}holds for any operator monotone function $f$ defined
 on interval $J$.

\end{corollary}
\begin{proof}
It suffices  to set $A=\sum_{j=1}^nf(A_j)$ and
$B=\sum_{j=1}^nf(B_j)$ in Theorem \ref{fri} to get the desired
inequality.
\end{proof}

Generally speaking, Theorem \ref{fri} does not hold for arbitrary
positive definite matrices $A$ and $B$. The reason for this lies
in the fact  that the inequality \eqref{bsd} is not true for
arbitrary positive numbers $a, b$. To see this, let $a=1, b=4,
r=2$.

Our next intention is to derive a result related to Theorem
\ref{fri} which holds for all positive definite matrices. Observe
that the inequality
\begin{align*}
\left((1+r)a-rb\right)^2-r^2(a-b)^2=(1+2r)a^2-2rab\leq {(a^2)}^{1+2r}(ab)^{-2r}=\left(a^{1+r}b^{-r}\right)^2
 \end{align*}
 yields an appropriate relation instead of \eqref{bsd}, for arbitrary positive numbers $a, b$ and $r\geq0$ or $r\leq -\frac{1}{2}$, as follows: \begin{align*}
\left((1+r)a-rb\right)^2\leq\left(a^{1+r}b^{-r}\right)^2+r^2(a-b)^2\,\,a,
b>0, r\geq0\ \ \mathrm{or} \ \ r\leq -\frac{1}{2}.
 \end{align*}
Note also that if $a=b$, then the equality holds.

Now, utilizing this inequality and  the same argument as in the
proof of Theorem \ref{fri}, i.e. the spectral theorem for positive
definite matrices,  we can accomplish the corresponding result.
 \begin{theorem}
Suppose that $A, B\in\mathcal{P}_n$ and $X\in\mathcal{M}_n$. Then
the inequality
 \begin{align}\label{young_weak}
\left\|(1+r)AX-rXB\right\|_2^2\leq
\left\|A^{1+r}XB^{-r}\right\|_2^2+r^2\left\|AX-XB\right\|_2^2
 \end{align}holds for $r\geq 0$ or $r\leq -\frac{1}{2}$.
\end{theorem}

\section{Reverse Young-type inequalities involving unitarily invariant norms}

\noindent It has been shown in \cite{heinz} that the inequality
\begin{align}\label{edc}
\|A^{1+r}XB^{1+r}\|\geq\|X\|^{-r}\|AXB\|^{1+r}
\end{align}
holds for $A,B\in\mathcal{P}_n$, $0\neq X\in\mathcal{M}_n$ and
$r\geq0$. Applying inequality \eqref{edc} yields the
relation
\begin{align}\label{23e1}
\|A^{1+r}XB^{-r}\|\geq\|AX\|^{1+r}\|XB\|^{-r},
\end{align}
where $r\geq0$, $A,B\in\mathcal{P}_n$ and $X\in\mathcal{M}_n$ with
$X\neq 0$.

Our next intention is to show that inequality \eqref{23e1} holds
for every unitarily invariant norm. This can be done by virtue of
inequality \eqref{edc1}. In fact, the following result is, in some
way, complementary to inequality \eqref{edc1}.
\begin{lemma}\label{djw}
Suppose that $A, B\in\mathcal{P}_n$, $X\in \mathcal{M}_n$ are such
that $X\neq0$. If $r\geq0$ or $r\leq-1$, then the inequality
\begin{align*}
|||AX|||^{1+r}\,|||XB|||^{-r}\leq|||A^{1+r}XB^{-r}|||
\end{align*}
holds for any unitarily invariant norm $|||\,.\,|||$.
\end{lemma}
\begin{proof}
First, let $r\geq0$. Set $\alpha=r+1$. Utilizing inequality
\eqref{edc1}, it follows that
\begin{align*}
|||AX|||=|||(A^\alpha)^{1\over\alpha}(XB^{1-\alpha})(B^\alpha)^{\alpha-1\over\alpha}|||&\leq
|||A^\alpha
XB^{1-\alpha}|||^{1\over\alpha}\,|||XB^{1-\alpha}B^\alpha|||^{\alpha-1\over\alpha}\\&=
|||A^\alpha
XB^{1-\alpha}|||^{1\over\alpha}\,|||XB|||^{\alpha-1\over\alpha},
\end{align*}
that is,
\begin{align*}
|||AX|||\,|||XB|||^{1-\alpha\over\alpha}\leq|||A^\alpha XB^{1-\alpha}|||^{1\over\alpha}.
\end{align*}
Hence,
\begin{align*}
|||AX|||^\alpha\,|||XB|||^{1-\alpha}\leq|||A^\alpha XB^{1-\alpha}|||,
\end{align*}
whence
\begin{align*}
|||AX|||^{1+r}\,|||XB|||^{-r}\leq|||A^{1+r} XB^{-r}|||.
\end{align*}
On the other hand, if $r\leq-1$, set $\alpha=-r$. By a similar
argument, we get the desired result.
\end{proof}
Applying Lemmas \ref{lemma11} and \ref{djw} yields the Young-type
inequality
\begin{align}\label{gla}
(1+r)|||AX|||-r|||XB|||\leq|||A^{1+r}XB^{-r}|||,
\end{align}
which holds for matrices $A, B\in\mathcal{P}_n, X\in
\mathcal{M}_n$ such that $X\neq0$ and $r\geq0$ or $r\leq-1$. It is
interesting that the inequality \eqref{gla} can be improved. But
first we have to improve the scalar inequality \eqref{prva}.

\begin{lemma}\label{lemma14}
Let $a, b>0$ and  $r\geq0$ or $r\leq-{1\over2}$. Then,
\begin{align}\label{prva_impr}
(1+r)a-rb+r(\sqrt{a}-\sqrt{b})^2\leq a^{1+r}b^{-r}.
 \end{align}
 \end{lemma}
\begin{proof}
Due to Lemma \ref{lemma11}, it follows that
\begin{align*}
(1+r)a-rb+r(\sqrt{a}-\sqrt{b})^2=-2r\sqrt{ab}+(1+2r)a\leq(\sqrt{ab})^{-2r}a^{1+2r}=a^{1+r}b^{-r}.
 \end{align*}
\end{proof}
Obviously, if $r\geq 0$, inequality \eqref{prva_impr} represents
an improvement of inequality \eqref{prva}. Finally, we give now an
improvement of matrix inequality \eqref{gla}.

\begin{theorem}\label{th_unitarily}
Let $A, B\in\mathcal{P}_n$, $X\in \mathcal{M}_n$ be such that
$X\neq0$ and let $r\geq0$. Then the inequality
\begin{align*}
(1+r)|||AX|||-r|||XB|||+r(\sqrt{|||AX|||}-\sqrt{|||XB|||})^2\leq|||A^{1+r}XB^{-r}|||
\end{align*}
holds for any unitarily invariant norm $|||\,.\,|||$.
\end{theorem}
\begin{proof}
\begin{align*}
(1+r)|||AX|||-r|||XB|||+r(\sqrt{|||AX|||}-\sqrt{|||XB|||})^2&\leq|||AX|||^{1+r}\,|||XB|||^{-r}\\&\qquad\qquad\qquad
\qquad(\textrm{by Lemma} \,\,\ref{lemma14})\\&\leq|||A^{1+r}XB^{-r}|||\\&\,\,\,\qquad\qquad
\qquad\qquad(\textrm{by Lemma}\,\, \ref{djw}).
\end{align*}
\end{proof}

\begin{remark}
It should be noticed here that the Theorem \ref{th_unitarily} is
also true in the case of $r\leq -\frac{1}{2}$. However, in this
case, the corresponding inequality is less precise than the
relation \eqref{gla} and does not represent its refinement.
\end{remark}

\section{Reverse Young-type inequalities related
to operator means }

The matrix Young inequality can be considered in a more general
setting. Namely, this inequality holds also for self-adjoint operators on a Hilbert space. The main objective of
this section is to derive inequalities which are complementary to
mean inequalities in (\ref{AGH_operator}), presented in the
Introduction.

The main tool in obtaining inequalities for  self-adjoint
operators on Hilbert spaces, is the following monotonicity property
for operator functions: If $X$ is a self-adjoint operator
with the spectrum $\mathrm{sp}(X)$, then
\begin{equation}
\label{spektar} f(t)\geq g(t),\ t\in \mathrm{sp}(X) \quad
\Longrightarrow \quad f(X)\geq g(X).
\end{equation}
For more details about this property the reader is referred to
\cite{Yuki}.

Since $A, B\in {\mathfrak B}^{+}(\mathcal{H})$, the expressions
$A\nabla_\nu B$ and $A\sharp_\nu B$ are also well-defined when
$\nu\in \mathbb{R}\setminus[0,1]$. In this case, we obtain reverse
of the second inequality in (\ref{AGH_operator}).
\begin{theorem}\label{AG_reverse}
If $A, B\in {\mathfrak B}^{+}(\mathcal{H})$ and $r\geq 0$ or
$r\leq -1$, then \begin{equation}\label{reverse_AG}
A\nabla_{-r}B\leq A\sharp_{-r}B.
\end{equation}
\end{theorem}

\begin{proof}
By virtue of Lemma \ref{lemma11}, it follows that
$f(x)=x^{-r}+rx-(1+r)\geq 0$, $x>0$. Moreover, since $B\in
{\mathfrak B}^{+}(\mathcal{H})$, it follows that
$A^{-\frac{1}{2}}BA^{-\frac{1}{2}}\in {\mathfrak
B}^{+}(\mathcal{H})$, that is,
$\mathrm{sp}(A^{-\frac{1}{2}}BA^{-\frac{1}{2}} )\in (0,\infty)$.

Thus, applying the monotonicity property (\ref{spektar}) to the above
function $f$, we have that
$$
\left(A^{-\frac{1}{2}}BA^{-\frac{1}{2}}
\right)^{-r}+rA^{-\frac{1}{2}}BA^{-\frac{1}{2}}-(1+r)I_\mathcal{H}\geq
0.
$$ Finally, multiplying both sides of this relation by
$A^{\frac{1}{2}}$, we have
$$
A^{\frac{1}{2}}\left(A^{-\frac{1}{2}}BA^{-\frac{1}{2}}
\right)^{-r}A^{\frac{1}{2}}+rB-(1+r)A\geq 0,
$$
and the proof is completed.
\end{proof}

If $A, B\in {\mathfrak B}^{+}(\mathcal{H})$ are such that $A\leq
B$, the expression $A!_{-r}B$ is well defined for $r\geq 0$.
Namely, due to operator monotonicity of the function
$h(x)=-\frac{1}{x}$ on
 $(0,\infty)$ (for more details, see \cite{Yuki}), $A\leq B$ implies that $B^{-1}\leq A^{-1}$, so that
 $(r+1)A^{-1}-rB^{-1}\in{\mathfrak
 B}^{+}(\mathcal{H})$. Therefore, the operator $A!_{-r}B=\left( (r+1)A^{-1}-rB^{-1}
 \right)^{-1}$ is well-defined for $r\geq0$.

 Now, we give the reverse of the first inequality in
 (\ref{AGH_operator}).

\begin{corollary} Let $A, B\in {\mathfrak B}^{+}(\mathcal{H})$
be such that $A\leq B$. If $r\geq 0$, then $A\sharp_{-r}B\leq
A!_{-r}B$.
\end{corollary}

\begin{proof}
Theorem \ref{AG_reverse} with operators $A$ and $B$ replaced by $A^{-1}$ and $B^{-1}$, respectively, follows that
 \begin{equation}\label{help3}A^{-1}\nabla_{-r}B^{-1}\leq A^{-1}\sharp_{-r}B^{-1}.
 \end{equation}

 Now, applying operator monotonicity of the function
 $h(x)=-\frac{1}{x}$,
 $x\in(0,\infty)$, to relation (\ref{help3}), we have that $\left( A^{-1}\sharp_{-r}B^{-1}\right)^{-1}\leq \left( A^{-1}\nabla_{-r}B^{-1}
 \right)^{-1}$. Finally,  the result follows since $\left(
 A^{-1}\sharp_{-r}B^{-1}\right)^{-1}=A\sharp_{-r}B$.

\end{proof}

Kittaneh \emph{et}.\emph{al}. obtained in \cite{debrecenkit} the
following relation (see also \cite{malezija}):
\begin{equation}
\label{debrecen}\begin{split} 2\max\{\nu,1-\nu\}(A\nabla B-A\sharp
B)&\geq A\nabla_\nu B-A\sharp_\nu B\\
&\geq 2\min\{\nu,1-\nu\}(A\nabla B-A\sharp B).\end{split}
\end{equation}
Clearly, the left inequality in (\ref{debrecen}) represents the
converse, while the right inequality represents the refinement of
arithmetic-geometric mean operator inequality in
(\ref{AGH_operator}).

Our next goal is to derive refinement of inequality
(\ref{reverse_AG}) which is, in some way, complementary to above
relations in (\ref{debrecen}). Clearly, this will be carried out
by virtue of Lemma \ref{lemma14}.

\begin{theorem}\label{tm}If $A, B\in {\mathfrak B}^{+}(\mathcal{H})$ and $r\geq
0$, then the following inequality holds
\begin{equation}\label{refinement_reverse_AG}
A\nabla_{-r}B+2r(A\nabla B-A\sharp B)\leq A\sharp_{-r}B.
\end{equation}
\end{theorem}
\begin{proof}
By virtue of Lemma \ref{lemma14}, it follows that
\begin{equation}\label{help4}(1+r)-rx+r(x-2\sqrt{x}+1)\leq x^{-r}\end{equation} holds for all $x>0$. Now,
applying the functional calculus, i.e. the property
(\ref{spektar}) to this scalar inequality, we have
$$(1+r)I_\mathcal{H}-rA^{-\frac{1}{2}}BA^{-\frac{1}{2}}+r(A^{-\frac{1}{2}}BA^{-\frac{1}{2}}-2\left( A^{-\frac{1}{2}}BA^{-\frac{1}{2}} \right)^{\frac{1}{2}}+I_\mathcal{H})\leq
\left( A^{-\frac{1}{2}}BA^{-\frac{1}{2}} \right)^{-r}.$$ Finally,
multiplying both sides of this operator inequality by
$A^{\frac{1}{2}}$, we obtain (\ref{refinement_reverse_AG}).
\end{proof}

\begin{corollary}
Let $A, B\in {\mathfrak B}^{+}(\mathcal{H})$ and $r>0$. Then,
$A\nabla_{-r}B= A\sharp_{-r}B$ if and only if $A=B$.
\end{corollary}

\begin{proof}
It follows from Theorem \ref{tm} and the fact that $A\nabla
B=A\sharp B$ if and only if $A=B$.
\end{proof}

\begin{remark} Having in mind that scalar inequality (\ref{help4}) holds also for $r\leq -\frac{1}{2}$ (see Lemma \ref{lemma14}), it follows that inequality
(\ref{refinement_reverse_AG}) holds also for $r\leq -\frac{1}{2}$.
However, if $r< -1$, relation (\ref{refinement_reverse_AG}) is
less precise than the original inequality (\ref{reverse_AG}) and
does not represent its refinement. On the other hand, it is
interesting to consider the case when $-1\leq r\leq -\frac{1}{2}$.
Namely, denoting $\nu=-r$, where $\frac{1}{2}\leq \nu\leq 1$,
(\ref{refinement_reverse_AG}) reduces to
$$
A\nabla_{\nu}B-2\nu(A\nabla B-A\sharp B)\leq A\sharp_{\nu}B,
$$
and this relation coincides with the converse of the
arithmetic-geometric mean inequality, that is, with the left
inequality in (\ref{debrecen}).

\end{remark}

\begin{remark}
In \cite{debrecenkit}, the authors considered operator version of
the classical Heinz mean, i.e., the operator
\begin{equation}
\label{Heinz_operator} H_\nu(A,B)=\frac{A\ \! \sharp_\nu \ \! B+
A\ \! \sharp_{1-\nu} \ \! B}{2},
\end{equation}where $A, B \in \mathfrak{B}^{+}(\mathcal{H})$, and $\nu\in [0,1]$.
Like in the real case, this mean interpolates in between
arithmetic and geometric mean, that is,
\begin{equation}
\label{heinz_interpolate} A\ \!\sharp \ \! B\leq H_\nu(A,B)\leq A\
\!\nabla \ \! B.
\end{equation}
On the other hand, since $A, B \in \mathfrak{B}^{+}(\mathcal{H})$,
the expression (\ref{Heinz_operator}) is also well-defined for
$\nu\in \mathbb{R}\setminus[0,1]$. Moreover, due to Theorem
\ref{AG_reverse}, we obtain the inequality
\begin{equation*}
H_{-r}(A,B)=\frac{A\sharp_{-r}B+A\sharp_{1+r}B}{2}\geq\frac{A\nabla_{-r}B+A\nabla_{1+r}B}{2}=A\nabla
B,\ r\geq 0 \ \mathrm{or} \ r\leq-1,
\end{equation*}
complementary to (\ref{heinz_interpolate}).
\end{remark}

In order to conclude this section, we mention yet another
inequality closely connected to the Young inequality. Namely, in
\cite{furuta2}, it has been shown the equivalence between the
Young inequality and the H\"{o}lder-McCarthy inequality which
asserts that
\begin{equation}
\label{holder_mcCarthy} \langle Ax,x \rangle^{-r}\leq \langle
A^{-r}x,x \rangle,\quad x\in \mathcal{H},\ \|x\|=1,
\end{equation}
holds for all $A \in \mathfrak{B}^{+}(\mathcal{H})$ and $r>0$ or
$r<-1$. If $-1<r<0$, then the sign of inequality in
(\ref{holder_mcCarthy}) is reversed.

Now, we give a refinement of the H\"{o}lder-McCarthy, once again
by exploiting Lemma \ref{lemma14}.

\begin{theorem}
Let $A\in {\mathfrak B}^{+}(\mathcal{H})$ and $r>0$. Then the
inequality
\begin{equation}
\label{holder_mccarthy_refinement} 0\leq 2r\left( 1-\langle
A^{\frac{1}{2}}x ,x \rangle \langle Ax,x \rangle^{-\frac{1}{2}}
\right)\leq \langle A^{-r}x,x \rangle \langle Ax,x \rangle^{r}-1
\end{equation}holds for any unit vector $x\in \mathcal{H}$.
\end{theorem}

\begin{proof}
By virtue of (\ref{help4}), it follows that the inequality
$2r(1-\sqrt{x})\leq x^{-r}-1$ holds for all $x>0$. Now, applying
the functional calculus to this inequality and the positive operator
$\lambda^{\frac{1}{r}}A$, $\lambda>0$, we have
$$
2r\left( I_\mathcal{H}
-\lambda^{\frac{1}{2r}}A^{\frac{1}{2}}\right)\leq
\lambda^{-1}A^{-r}-I_\mathcal{H}.
$$
Further, fix a unit vector $x\in \mathcal{H}$. Then we have
$$
2r\left( 1 -\lambda^{\frac{1}{2r}}\langle
A^{\frac{1}{2}}x,x\rangle\right)\leq \lambda^{-1}\langle
A^{-r}x,x\rangle-1.
$$
Finally, putting $\lambda=\langle Ax,x\rangle^{-r}$ in the last
inequality, we obtain second inequality in
(\ref{holder_mccarthy_refinement}). Clearly, the first inequality
sign in (\ref{holder_mccarthy_refinement}) holds due to
(\ref{holder_mcCarthy}) since $\langle A^{\frac{1}{2}}x,x
\rangle\leq \langle Ax,x \rangle^{\frac{1}{2}}$.
\end{proof}

\begin{remark}
Since relation (\ref{help4}) holds for $r\leq -\frac{1}{2}$, it
follows that the second inequality in
(\ref{holder_mccarthy_refinement}) holds also for $r\leq
-\frac{1}{2}$. Clearly, the case of $r<-1$ is not interesting
since in this case we obtain less precise relation than the
original H\"{o}lder-McCarthy inequality (\ref{holder_mcCarthy}).
On the other hand, the case of $-1<r<-\frac{1}{2}$ yields a
converse of (\ref{holder_mcCarthy}).
\end{remark}
\section{Reverse Young-type inequalities for the trace and the determinant}

In this section we derive some Young-type inequalities for the trace
and the determinant of a matrix. The starting point for this
direction is already used Lemma \ref{lemma14}.

In \cite{young2}, Kittaneh and Manasrah obtained the inequality
\begin{equation}\label{kittaneh_trace}
\mathrm{tr} \left|A^{\nu}B^{1-\nu} \right|+
r_0\left(\sqrt{\mathrm{tr} A}-\sqrt{\mathrm{tr} B} \right)^{2}\leq
\mathrm{tr} \left( \nu A+ (1-\nu)B \right),
\end{equation}
which holds for positive semidefinite matrices $A, B\in
\mathcal{M}_n$, $0\leq \nu\leq 1$, and $r_0=\min\{\nu, 1-\nu\}$.

By virtue of Lemma \ref{lemma14}, we can accomplish the inequality
complementary to (\ref{kittaneh_trace}). To do this, we also need
the following inequality regarding singular values of complex
matrices:
\begin{equation}
\label{singular_values} \sum_{j=1}^n s_j(A)s_{n-j+1}(B)\leq
\sum_{j=1}^n s_j(AB)\leq \sum_{j=1}^n s_j(A)s_j(B).
\end{equation}
Now, we have the following result:
\begin{theorem}\label{tmtr}
If $A, B\in \mathcal{P}_n$ and $r\geq 0$, then the following
inequality holds:
\begin{equation}
\label{reverse_trace} \mathrm{tr}((1+r)A-rB)\leq
\mathrm{tr}\left|A^{1+r}B^{-r} \right|-r\left( \sqrt{\mathrm{tr}
A} - \sqrt{\mathrm{tr} B}\right)^{2}.
\end{equation}
\end{theorem}
\begin{proof}

By Lemma \ref{lemma14}, we have
$$
(1+r)s_j(A)-rs_{n-j+1}(B)\leq
s_j^{1+r}(A)s_{n-j+1}^{-r}(B)-r\left(\sqrt{s_j(A)}-\sqrt{s_{n-j+1}(B)}\right)^{2},
$$
for $j=1,2,\ldots, n$.

Now, utilizing the above inequality and (\ref{singular_values}),
as well as the properties of the trace functional, it follows that
\begin{equation*}
\begin{split}
\mathrm{tr}((1+r)A-rB)&=(1+r)\mathrm{tr}A-r \mathrm{tr} B\\
&=\sum_{j=1}^n \left( (1+r)s_j(A)-rs_{n-j+1}(B) \right)\\
&\leq \sum_{j=1}^n
s_j^{1+r}(A)s_{n-j+1}^{-r}(B)\\
&\qquad-r\sum_{j=1}^n\left(
s_j(A)+ s_{n-j+1}(B)-2 \sqrt{s_j(A)s_{n-j+1}(B)}\right)\\
&=\sum_{j=1}^n s_j(A^{1+r})s_{n-j+1}(B^{-r})\\
&\qquad-r\left(\mathrm{tr} A
+ \mathrm{tr} B-2\sum_{j=1}^n\sqrt{s_j(A)s_{n-j+1}(B)} \right)\\
&\leq \sum_{j=1}^n s_j(A^{1+r}B^{-r})-r\left(\mathrm{tr} A +
\mathrm{tr} B-2\sum_{j=1}^n\sqrt{s_j(A)s_{n-j+1}(B)} \right).
\end{split}
\end{equation*}
Moreover, by virtue of the well-known Cauchy-Schwarz inequality,
we have
\begin{equation*}
\begin{split}
\sum_{j=1}^n\sqrt{s_j(A)s_{n-j+1}(B)}&\leq \left(\sum_{j=1}^n
s_j(A) \right)^{\frac{1}{2}}\left(\sum_{j=1}^n
s_{n-j+1}(B) \right)^{\frac{1}{2}}=\sqrt{\mathrm{tr} A \mathrm{tr} B},\\
\end{split}
\end{equation*}
so that
$$
\mathrm{tr}((1+r)A-rB)\leq \mathrm{tr}\left|A^{1+r}B^{-r}
\right|-r\left( \mathrm{tr} A + \mathrm{tr} B-2\sqrt{\mathrm{tr} A
\mathrm{tr} B}\right).
$$
This completes the proof.
\end{proof}

\begin{remark}
Although the proof of Theorem \ref{tmtr} seems to be very
interesting, it can be accomplished in a much simpler way, if we
take into account Theorem \ref{th_unitarily}. Namely, considering
Theorem \ref{th_unitarily} with $X=I_n$ and with the trace norm
$\|\cdot\|_1$, that is, $\|A\|_1=\sum_{i=1}^n s_j(A)=\mathrm{tr}
|A|$, it follows that
$$
(1+r)\|A\|_1-r\|B\|_1+r(\sqrt{\|A\|_1}-\sqrt{\|B\|_1})^2\leq\|A^{1+r}B^{-r}\|_1.
$$
Now, since $A, B\in \mathcal{P}_n$, it follows that
$\|A\|_1=\mathrm{tr} A$ and $\|B\|_1=\mathrm{tr} B$, that is,
$(1+r)\|A\|_1-r\|B\|_1=\mathrm{tr}((1+r)A-rB)$, so we retain the
inequality \eqref{reverse_trace}.
\end{remark}

Our next intention is to obtain an analogous reverse relation for
the determinant of a matrix. In \cite{young2}, the authors
obtained inequality
$$
\mathrm{det} (A^{\nu}B^{1-\nu})+r_0^{n}\mathrm{det} (2A\nabla
B-2A\sharp B)\leq \mathrm{det} (\nu A+(1-\nu)B),
$$
where $0\leq \nu\leq 1$, $r_0=\min\{\nu, 1-\nu\}$, and $A,B$ are
positive definite matrices. The corresponding complementary result
can also be established by virtue of Lemma \ref{lemma14}.

\begin{theorem}\label{tmdet} Let $r\geq 0$ and let $A, B\in \mathcal{P}_n$ be
such that $A\geq \frac{r}{r+1}B$. Then the following inequality
holds:
\begin{equation}
\label{determinat_reverse} \mathrm{det}\left( (1+r)A-rB
\right)\leq \mathrm{det} \left(A^{r+1}B^{-r}
\right)-r^{n}\mathrm{det}\left( 2A\nabla B-2A\sharp B \right).
\end{equation}
\end{theorem}

\begin{proof} The starting point is Lemma \ref{lemma14} with
$a=s_j\left(B^{-\frac{1}{2}}AB^{-\frac{1}{2}} \right)$ and $b=1$,
i.e. the inequality
$$
s_j^{r+1}\left(B^{-\frac{1}{2}}AB^{-\frac{1}{2}} \right)\geq
(1+r)s_j\left(B^{-\frac{1}{2}}AB^{-\frac{1}{2}}
\right)-r+r\left(s_j^{\frac{1}{2}}\left(B^{-\frac{1}{2}}AB^{-\frac{1}{2}}
\right)-1 \right)^{2}.
$$
Furthermore, since $A\geq \frac{r}{r+1}B$, it follows that
$B^{-\frac{1}{2}}AB^{-\frac{1}{2}}\geq \frac{r}{r+1}I_n$, which
means that $s_j\left(B^{-\frac{1}{2}}AB^{-\frac{1}{2}} \right)\geq
\frac{r}{r+1}$. Consequently, we have that
$$
(1+r)s_j\left(B^{-\frac{1}{2}}AB^{-\frac{1}{2}} \right)-r\geq 0.
$$
Hence, by virtue of the above two relations and the well-known
determinant properties, we have
\begin{equation*}
\begin{split}
\mathrm{det}\left(B^{-\frac{1}{2}}AB^{-\frac{1}{2}}
\right)^{r+1}&=\prod_{j=1}^n
s_j^{r+1}\left(B^{-\frac{1}{2}}AB^{-\frac{1}{2}} \right)\\
&\geq
\prod_{j=1}^n\Bigg[(1+r)s_j\left(B^{-\frac{1}{2}}AB^{-\frac{1}{2}}
\right)-r\\
&\qquad+r\left(s_j^{\frac{1}{2}}\left(B^{-\frac{1}{2}}AB^{-\frac{1}{2}}
\right)-1 \right)^{2}\Bigg]\\
&\geq
\prod_{j=1}^n\left[(1+r)s_j\left(B^{-\frac{1}{2}}AB^{-\frac{1}{2}}
\right)-r\right]\\
&\qquad +
r^{n}\prod_{j=1}^n\left[\left(s_j^{\frac{1}{2}}\left(B^{-\frac{1}{2}}AB^{-\frac{1}{2}}
\right)-1 \right)^{2}\right]\\
&=\mathrm{det}\left((1+r)B^{-\frac{1}{2}}AB^{-\frac{1}{2}}-rI_n
\right)\\
&\qquad+r^{n}\mathrm{det}\left(\left(B^{-\frac{1}{2}}AB^{-\frac{1}{2}}
\right)^{\frac{1}{2}}-I_n \right)^{2}.
\end{split}
\end{equation*}
Finally, multiplying both sides of the obtained inequality by
$\mathrm{det}(B^{\frac{1}{2}} )$ and utilizing the well-known
Binet-Cauchy theorem, we obtain (\ref{determinat_reverse}), as
claimed.
\end{proof}

\section{ Reverses of the Young inequality dealing with singular values }

\noindent Let $x=(x_1,\ldots, x_n), y=(y_1,\ldots,
y_n)\in\mathbb{R}^n$ be such that $0\leq x_1\leq\cdots\leq x_n $
and $0\leq y_1\leq\cdots\leq y_n$. Then $x$ is said to be log
majorized by $y$, and denoted by $x\prec_{\log} y$, if
\begin{align*}
\prod_{j=1}^kx_j\leq \prod_{j=1}^ky_j\qquad(1\leq k< n)\qquad\textrm{and}\qquad \prod_{j=1}^nx_j= \prod_{j=1}^ny_j.
\end{align*}
For $X\in\mathcal{M}_n$ and $k=1, \ldots, n$, the $k$-th compound
of $X$ is defined as the ${{n\choose k}\times{n\choose k}}$
complex matrix $C_k(X)$, whose entries are defined by
$C_k(X)_{r,s}={\rm det} X[(r_1, r_2,\cdots, r_k)|(s_1, s_2,\cdots,
s_k)]$, where $(r_1, r_2,\cdots, r_k),(s_1, s_2,\cdots, s_k)\in
P_{k,n}=\{(x_1,\cdots, x_k)\,\,|\,\, 1\leq x_1 < \cdots< x_k \leq
n\}$ are arranged in a lexicographical order and $(r_1,
r_2,\cdots, r_k)$ and $(s_1, s_2,\cdots, s_k)$ are the $r$-th and
$s$-th element in $P_{k,n}$, respectively. $X[r,s]$ is the
$k\times k$ matrix that contains the elements in the intersection
of rows $(r_1, r_2,\cdots, r_k)\in P_{k,n}$ and columns $(s_1,
s_2,\cdots, s_k)\in P_{k,n}$ (for more details, see
\cite{merris}).
 For example, if $n=3$ and $k=2$, then $(1,2),(1,3)$ and $(2,3)$ are the first, the second and the third element of $P_{k,n}$, respectively.
 So,
\begin{align*}
C_2(X)=\left(\begin{array}{ccc}
 {\rm det} X[1,2|1,2]&{\rm det} X[1,2|1,3]&{\rm det} X[1,2|2,3]\\
 {\rm det} X[1,3|1,2]&{\rm det} X[1,3|1,3]&{\rm det} X[1,3|2,3]\\
 {\rm det} X[2,3|1,2]&{\rm det} X[2,3|1,3]&{\rm det} X[2,3|2,3]
 \end{array}\right).
 \end{align*}
 In a general case, for $A, B\in\mathcal{M}_n$, we have
\begin{align}\label{sese}
 C_k(AB)=C_k(A)C_k(B)\qquad \textrm{and} \qquad s_1(C_k(A))=\prod_{j=1}^ks_j(A)\,\,(1\leq k\leq n).
 \end{align}
Finally we use the corresponding ideas from \cite{lie} to present
our last result.
\begin{theorem}
Suppose that $A, B\in\mathcal{P}_n$ and $ X\in \mathcal{M}_n$. If
$r\geq 0$, then
\begin{itemize}
\item[(i)]$s(A^{1+r}XB^{1+r})\succ_{\log}s^{1+r}(AXB)s^{-r}(X),$
 \item[(ii)]$s(A^{1+r}XB^{-r})\succ_{\log}s^{1+r}(AX)s^{-r}(XB).$
\end{itemize}
\end{theorem}
\begin{proof}
${\rm(i)}$ Let $C_k(X)\in\mathbb{C}_{{n\choose k}\times{n\choose
k}}$ denote the $k$-th component of $X$, $1\leq k\leq n$. Then, we
have
\begin{align*}
\prod_{i=1}^{k}s_{i}(A^{1+r}XB^{1+r})&=s_1(C_k(A^{1+r}XB^{1+r}))\qquad \textrm{by } \eqref{sese}\\&=s_1(C_k(A)^{1+r}C_k(X)C_k(B)^{1+r})\qquad \textrm{by } \eqref{sese}\\&\geq
s_1^{-r}(C_k(X))s_1^{1+r}(C_k(AXB))\qquad \textrm{(by inequality}\, \eqref{edc})\\&=\prod_{i=1}^{k}s_{i}^{-r}(X)\prod_{i=1}^{k}s_{i}^{1+r}(AXB).
\end{align*}
Moreover, if $k=n$, we have
 \begin{align*}
\prod_{i=1}^{n}s_{i}(A^{1+r}XB^{1+r})=|{\rm det}(A^{1+r}XB^{1+r})|=({\rm det}A)^{1+r}|{\rm det}X|({\rm det}B)^{1+r}
\end{align*}
and
\begin{align*}
\prod_{i=1}^{n}s_{i}(X)^{-r}\prod_{i=1}^{n}s_{i}(AXB)^{1+r}=|{\rm det}X^{-r}|\,|{\rm det}(AXB)^{1+r}|=({\rm det}A)^{1+r}|{\rm det}X|({\rm det}B)^{1+r}.
\end{align*}
${\rm(ii)}$ The second conclusion can be accomplished by a similar
argument as in  ${\rm(i)}$ and by utilizing the inequality
\eqref{23e1}.
\end{proof}

\bibliographystyle{amsplain}

\begin{thebibliography}{99}

\bibitem{ando}T. Ando, \textit{Matrix Young inequality}, J. Oper. Theory Adv. Appl., \textbf{75} (1995), 33--38.

\bibitem{conde} C. Conde, \textit{Young type inequalities for positive operators}, Ann. Funct. Anal. \textbf{4} (2013), no. 2, 144--152.

\bibitem{fuji} J.I. Fujii, \textit{An external version of the Jensen operator inequality}, Sci. Math. Japon. Online, e-2011,
 59--62.
 
\bibitem{FUR} S. Furuichi and N. Minculete, \textit{Alternative reverse inequalities for Young's inequality}, J. Math. Inequal. \textbf{5} (2011), no. 4, 595--600.

\bibitem{furuta}T. Furuta and M. Yanagida, \textit{Generalized means and convexity of inversion for positive operators},
Amer. Math. Monthly, \textbf{105} (1998), 258--259.

\bibitem{furuta2} T. Furuta, \textit{The H\"{o}lder-McCarthy and the Young inequalities are equivalent for Hilbert space operators}, Amer. Math. Monthly, \textbf{108} (2001), 68--69.


\bibitem{hirzallah} O. Hirzallah and F. Kittaneh, \textit{Matrix Young inequalities for the Hilbert-Schmidt norm}, Linear Algebra Appl. \textbf{308} (2000), 77--84.

\bibitem{heinz} E. Heinz, \textit{Beitra\"{g}e zur St\"{o}rungstheoric der Spektralzerlegung}, Math. Ann.
\textbf{123} (1951), 415--438.

\bibitem{horn} R.A. Horn and C.R. Johnson, \textit{Topics in Matrix Analysis}, Cambridge University Press, 1991.

\bibitem{young2} F. Kittaneh and Y. Manasrah, \textit{Improved Young and Heinz inequalities for matrices}, J. Math. Anal. Appl.
\textbf{361} (2010), 262--269.

\bibitem{kittaneh}F. Kittaneh, \textit{Norm inequalities for fractional powers of positive operators}, Lett. Math. Phys. \textbf{27} (1993), 279--285.

\bibitem{debrecenkit} F. Kittaneh, M. Krni\' c, N. Lovri\v cevi\' c and  J. Pe\v cari\' c, \textit{Improved arithmetic-geometric and Heinz means
inequalities for Hilbert space operators}, Publ. Math. Debrecen
\textbf{ 80} (2012), no. 3-4, 465--478.

\bibitem{kosa} H. Kosaki, \textit{Arithmetic-geometric mean and related inequalities for operators}, J. Funct. Anal.
\textbf{156} (1998), 429--451.

\bibitem{malezija} M. Krni\' c, N. Lovri\v cevi\' c and J. Pe\v cari\' c, \textit{Jensen's
operator and applications to mean inequalities for operators in
Hilbert space}, Bull. Malays. Math. Sci. Soc. (2) \textbf{ 35}
(2012), no. 1, 1--14.

\bibitem{MAN1} S.M. Manjegani, \textit{H\"older and Young inequalities for the trace of operators}, Positivity \textbf{11} (2007), no. 2, 239--250. MR2321619 (2008h:47024) Add to clipboard

\bibitem{MAN2} S.M. Manjegani, \textit{ Spectral dominance and Young's inequality in type III factors}, J. Inequal. Pure Appl. Math. \textbf{7} (2006), no. 3, Article 82, 8 pp.

\bibitem{merris}R. Merris, \textit{Multilinear Algebra}, Gordan and Breach Science Publishers, Amsterdam 1997.


\bibitem{Yuki} J. Pe\v cari\' c, T. Furuta, J. Mi\' ci\' c Hot, Y. Seo, \textit{Mond-Pe\v cari\' c Method in Operator Inequalities}, Element,
Zagreb, 2005.

\bibitem{lie} T.Y. Tam, \textit{Heinz-Kato's inequalities for semisimple Lie groups}, J. Lie Theory. {\textbf18} (2008), no. 4, 919--931.

\bibitem{TOM1} M. Tominaga, \textit{Specht's ratio in the Young inequality}, Sci. Math. Jpn. \textbf{55} (2002), no. 3, 583--588.

\bibitem{TOM2} M. Tominaga, \textit{Specht's ratio and logarithmic mean in the Young inequality}, Math. Inequal. Appl. \textbf{7} (2004), no. 1, 113--125.

\bibitem{Zhang1} F. Zhang, \textit{Matrix Theory}, Springer-Verlag New York, 2011.
\end{thebibliography}

\end{document}